\documentclass[11pt]{amsart}
\usepackage{amsmath,amsthm,amsfonts,amssymb,amscd,graphicx, xypic, enumerate}
\usepackage[all]{xy}
\usepackage{mathrsfs}
\usepackage{xspace}
\usepackage{graphics}
\usepackage{graphicx}
\usepackage[colorlinks=true,hyperindex]{hyperref} 

\theoremstyle{plain}

\newtheorem{thm}{Theorem}
\newtheorem{lem}[thm]{Lemma}
\newtheorem{cor}[thm]{Corollary}
\newtheorem{prop}[thm]{Proposition}

\newtheorem{thma}{Theorem}
\newtheorem{cora}[thma]{Corollary}
\newtheorem{conj}[thm]{Conjecture}

\theoremstyle{definition}
\newtheorem{defn}[thm]{Definition}

\newtheorem{rem}[thm]{Remark}

\numberwithin{thm}{section}

\newcommand{\comment}[1]{}

\newcommand{\Om}{\Omega}
\newcommand{\coker}{\operatorname{coker}}

\newcommand{\Hom}{\operatorname{Hom}}
\newcommand{\Ext}{\operatorname{Ext}}

\newcommand*{\longhookrightarrow}{\ensuremath{\lhook\joinrel\relbar\joinrel\rightarrow}}

\newcommand{\Eext}{\widehat{\Ext}}
\newcommand{\ed}{{\rm ext.deg }}

\newcommand{\im}{\operatorname{im}}
\newcommand{\ql}{\operatorname{ql}}

\newcommand{\fed}{{\rm{fed}}}

\renewcommand{\to}{\longrightarrow}

\newcommand{\StHom}{\underline{\operatorname{Hom}}}

\newcommand{\D}{\operatorname{D}}

\newcommand{\Lmod}{\operatorname{\Lambda \textrm{-}mod}}

\begin{document}

\title{Vanishing of self-extensions over symmetric algebras}

\author{Kosmas Diveris \ and \ Marju Purin}\address{Department of Mathematics, Statistics, and Computer Science \\St.\ Olaf College\\ Northfield\\ MN 55057\\ USA}\email{diveris@stolaf.edu}
\address{Department of Mathematics, Statistics, and Computer Science \\St.\ Olaf College\\ Northfield\\ MN 55057\\ USA}\email{purin@stolaf.edu}

\maketitle 

\begin{abstract} We study self-extensions of modules over symmetric artin algebras. We show that non-projective modules with eventually vanishing self-extensions must lie in AR components of stable type $\mathbb{Z}A_{\infty}$. Moreover, the degree of the highest non-vanishing self-extension of these modules is determined by their quasilength. This has implications for the Auslander-Reiten Conjecture and the Extension Conjecture. \end{abstract}

\section{Introduction}

In this article we study finitely generated modules with eventually vanishing self-extensions over a symmetric artin algebra $\Lambda$ via its Auslander-Reiten quiver. This investigation is motivated by two questions.  The first is to examine which algebras satisfy the Generalized Auslander-Reiten Condition.

\begin{defn} \label{cond:GARC} A ring $R$ is said to satisfy the {\it Generalized Auslander-Reiten Condition} (GARC) if for each $R$-module $M$ with $\Ext_R^i(M,M \oplus R) = 0$ for all $i > n$ the projective dimension of $M$ is at most $n$.  \end{defn}  

In the special case when $n=0$ one obtains the Auslander-Reiten Condition (ARC) for the ring $R$.  The Auslander-Reiten Conjecture, still an open question, asserts that all artin algebras satisfy the Auslander-Reiten Condition \cite{AR}.  It is clear that any algebra satisfying (GARC) also satisfies (ARC), but it is known that these conditions are not equivalent. It is also known that there do exist artin algebras which do not satisfy (GARC) \cite{Schulz}. 

The second motivating question has been called the Extension Conjecture.  In \cite{ILP}, K. Igusa, S. Liu, and C. Paquette prove the Strong No Loop Conjecture, originally stated in \cite{Z}, for finite-dimensional elementary algebras and conjecture that the following stronger statement holds:

\begin{conj}[\cite{ILP}] Let $S$ be a simple module over an artin algebra with $\Ext^1(S,S) \neq 0$. Then we have $\Ext^i(S,S) \neq 0$ for infinitely many integers $i$. 
\end{conj}  

This statement again concerns modules with eventually vanishing self-extensions.  In fact, the Extension Conjecture holds for self-injective algebras satisfying (GARC).  Indeed, if $\Lambda$ is self-injective and satisfies (GARC) then any simple $\Lambda$-module $S$ with $\Ext^i(S,S) = 0$ for all $i \gg 0$ must be projective and thus $\Ext^1(S,S) = 0$.

In this paper, we describe which AR components may contain modules with eventually vanishing self-extensions. A similar question has been considered before.  M. Hoshino has shown that a module with trivial self-extensions in all positive even degrees must appear on the boundary of a component of stable type $\mathbb{Z}A_{\infty}$ \cite[Theorem 1.5]{Hoshino}.  Our main result is as follows: 

\begin{thma} \label{thm:main} Assume that $\Lambda$ is a symmetric artin algebra.  Let $\mathcal{C}$ be an AR component containing a non-projective module with eventually vanishing self-extensions.  Then the stable part of $\mathcal{C}$ is of the form $\mathbb{Z}A_{\infty}$.  
\end{thma}

Moreover, we show that all modules in $\mathcal{C}$ have eventually vanishing self-extensions and describe the highest degree in which a module in $\mathcal{C}$ has a non-trivial self-extension.  This implies that when (GARC) does not hold for a symmetric algebra, there must be an entire $\mathbb{Z}A_{\infty}$ component consisting of modules with eventually vanishing self-extensions. 

Theorem A gives a reduction criterion for both the  Generalized Auslander-Reiten Condition and the Extension Conjecture. We obtain the following corollary as an application of our main result.

\begin{cora} Let $\Lambda$ be a symmetric artin algebra. Then both the Generalized Auslander-Reiten Condition and the Extension Conjecture hold in the following cases:

\begin{enumerate}
\item $\Lambda$ has no AR component of stable type $\mathbb{Z}A_{\infty}$.
\item $\Lambda$ is of wild tilted type.
\item $\Lambda$ is the trivial extension of an iterated tilted algebra.
\end{enumerate}
\end{cora}

Since the Extension Conjecture implies the Strong No Loop Conjecture, we also obtain the validity of the latter in the above cases. We remark that unlike \cite{ILP} we do not require $\Lambda$ to contain a field.

We now give an outline of the article. In \autoref{sec:vanishing} we set up notation and collect a series of technical results on the vanishing of self-extensions. Theorem\autoref{thm:tubes} and\autoref{thm:bands} establish the condition (GARC) for stable AR components that are $\tau$-periodic, or that have finitely many $\tau$-orbits and a slice without cycles. In \autoref{sec:Omega_perfect} we make use of the notion of $\Om$-perfect AR components which were introduced by E. Green and D. Zacharia \cite{GZ1}. We prove the analog to Hoshino's result for eventually $\Omega$-perfect AR components. In \autoref{sec:not_Omega_perfect} we treat components of $\Gamma_{\Lambda}$ which are are not eventually $\Omega$-perfect. We show that all non-projective modules in these components have non-vanishing self-extensions in infinitely many degrees. Finally, in \autoref{sec:applications} we prove the Generalized AR Condition and the Extension Conjecture for some classes of symmetric algebras, including those of wild tilted type.

\section{Vanishing of Self-extensions} \label{sec:vanishing}
% \section{Preliminaries} \label{sec:prelim}

In this paper we consider symmetric algebras. Recall that an Artin algebra $\Lambda$ is symmetric if $\Lambda \cong \D(\Lambda)$ as $\Lambda$-bimodules where $\D$ denotes the standard duality. We assume throughout that all modules are finitely generated left $\Lambda$-modules. The Auslander-Reiten translate of a module $M$ is denoted by $\tau(M)$.  For symmetric algebras $\tau=\Omega^2$ where $\Omega$ denotes the syzygy operator. All finitely generated modules over $\Lambda$ can be organized into a combinatorial device called the Auslander-Reiten quiver which we denote by $\Gamma_{\Lambda}$. 
Furthermore, the {\it stable AR quiver} of $\Lmod$ is obtained by removing all of the vertices that correspond to projective modules from the AR quiver of $\Lmod$.
For background on AR theory we refer the reader to \cite{ARS}.  

We now turn to cohomology in the stable module category.  Over a self-injective algebra, every module has a complete resolution, i.e. there is an acyclic complex of projective modules $(P. , \delta .)$ where $\im \delta_0 \cong M$.  For each $i \in \mathbb{Z}$, we set $\Omega^i(M) = \im \delta_i$.   For each $i \in \mathbb{Z}$, we define the stable cohomolgy of $M$ with coefficients in $N$ to be $\Eext^i(M,N) = \StHom(\Omega^i(M),N)$.  Note that according to the definitions above, the truncated complex $P_{\geq 0}$ gives a projective resolution of $M$, so that for all $i \geq 0$, $\Omega^i(M)$ is the $i^{th}$ syzygy of $M$. From this it follows that for $i \geq 1$ we have $\Eext^i(M,N) = \Ext^i(M,N)$. In degree zero however, there is an important difference between these cohomologies.  If $M$ is a nonzero module, then the identity map is a nonzero morphism in $\Hom_{\Lambda}(M,M)$, but $\StHom_{\Lambda}(M,M) =0$ if and only if the module $M$ is projective. 

We record the following dimension shift isomomorphism that we employ frequently:
\begin{align} \label{eq:dimshift} \Eext^i(M,N) \cong \Eext^{i-m+n}(\Omega^m(M), \Omega^n(N)) \end{align} which holds for all $i, m, n \in \mathbb{Z}$.   
One advantage of using stable cohomology is that dimension shifting holds for all indices (when using classical cohomology, dimension shifting only holds for extensions of positive degrees).  Another advantage of using stable cohomology in our setting is that the projective modules are characterized by the vanishing of their stable endomorphism groups.  

\comment{

While we will not make use of this fact, we do point out that when $\Lambda$ is symmetric, combining equations (\ref{eq:AR.formula}) and (\ref{eq:dimshift}) gives \begin{align} \D(\Ext^i(M,N) \cong \D\StHom(\Omega^iM,N) \cong \Ext^1(N, \Omega^{i+2}M) \cong \Eext^{-i-1}(N,M). \end{align} In particular, this gives $\Eext^i(M,M) = 0$ for all $i \gg 0$ if and only if $\Eext^i(M,M) = 0$ for all $i \ll 0$.  Also, if $\Eext^i(M,M) = 0$ for all $i>0$, then the graded algebra $\Eext^*(M,M) = \Eext^0(M,M) \oplus \Eext^{-1}(M,M) = \StHom(M,M) \oplus \D(\StHom(M,M))$ is the trivial extension of the stable endomorphism algebra of $M$.

Actually, is it true that $\Eext^*(M,M) = \Eext^{\geq 0}(M,M) \ltimes \D\Eext^{\geq 0}(M,M)$? }

%%%%%%%%%%%%%%%%%%%%
%%%%%%%%%%%%%%%%%%%%

\subsection{Vanishing of Self-Extensions}

In this section we record several results concerning the vanishing of cohomology that we need in the sequel. Our first lemma is used repeatedly.  

\begin{lem} \label{lem:indices} Let  $0 \rightarrow \Om^2 M \rightarrow N \rightarrow M  \rightarrow 0$ be a short exact sequence in $\Lmod$. 

\begin{enumerate}[(i)] %
\item \label{lem:cov.indices}  If $X$ is an $\Lambda$-module such that  $\Eext_{\Lambda}^k(X,M) = 0$ for $k>n$, then we have $\Eext_{\Lambda}^k(X,N) = 0$ for $k > n+2$, and $\Eext_{\Lambda}^{n+2}(X,N) \cong \Eext_{\Lambda}^{n}(X,M)$. 
\item \label{lem:con.indices}  If $X$ is an $\Lambda$-module such that $\Eext_{\Lambda}^k(M,X) = 0$ for $k>n$, then we have  $\Eext_{\Lambda}^k(N,X) = 0$ for $k > n$, and $\Eext_{\Lambda}^n(N,X) \cong \Eext_{\Lambda}^n(M,X)$. %
\end{enumerate} 
\end{lem}

\begin{proof}  (i) Consider the long exact sequence in $\Eext^*(X,-)$: \begin{align*} \cdots \rightarrow \Eext_{\Lambda}^{k-1}(X, M) \rightarrow \Eext_{\Lambda}^{k}(X, \Om^2M) \rightarrow \Eext_{\Lambda}^{k}(X, N) \rightarrow \Eext_{\Lambda}^{k}(X, M) \rightarrow \cdots \end{align*} Shifting dimensions, we see that the second term listed is isomorphic to $\Eext_{\Lambda}^{k-2}(X,M)$.  Therefore the second and fourth terms vanish when $k>n+2$, so the third must also vanish.  This demonstrates that the stated vanishing holds.  For the isomorphism, we consider the piece of the above sequence when $k=n+2$.  In this case the first and last terms are zero and the dimension shift gives that the second term is isomorphic to $\Eext^n(X,M)$. The middle map now provides the desired isomorphism. 

The proof of (ii) is analogous to that of (i).
\end{proof}

We find it useful to introduce the notion of an extension degree of a module as a means of describing the vanishing of self-extensions in AR components.

\begin{defn} For a $\Lambda$-module $M$ we define the  extension degree of $M$ to be $\ed(M) = \sup\{ i \ | \ \Ext^i(M,M) \neq 0 \}$. \end{defn}

An application of the above lemma gives the following.

\begin{lem} \label{lem:ed.of.middle.term} Let $M$ be a non-projective $\Lambda$-module with $\ed(M) = m$. Let $0 \rightarrow \Om^2 M \rightarrow N \rightarrow M \rightarrow 0$ be an exact sequence. Then we have $\ed(N) = m+2$ and an isomorphism $\Eext_{\Lambda}^{m+2}(N,N) \cong \Eext_{\Lambda}^m(M,M)$.
\end{lem} 

\begin{proof}   We have $\Eext_{\Lambda}^m(M,M)\neq 0$.   First invoke part (\ref{lem:cov.indices}) of Lemma\autoref{lem:indices}, with $X = M$. It 
provides an isomorphism $f: \Eext_{\Lambda}^{m+2}(M,N) \cong \Eext^{m}(M, M)$ and $\Eext^k(M,N) = 0$ for $k> m+2$. 

Next, we apply part (\ref{lem:con.indices}) of Lemma\autoref{lem:indices}.  This time we take $X = M$ and use the vanishing obtained in the previous paragraph.  It provides an isomorphism $g: \Eext_{\Lambda}^{m+2}(N,N) \cong \Eext_{\Lambda}^{m+2}(M, N)$ and $\Eext_{\Lambda}^k(N,N) = 0$ for $k> m+2$. 

In particular, we see that $\ed(N) \leq m+2$.  Composing the isomorphisms $f$ and $g$ gives the desired isomorphism between  $\Eext_{\Lambda}^{m+2}(N,N)$ and $\Eext_{\Lambda}^m(M,M)$.  Since the latter is nontrivial, this gives $\ed(N) \geq m+2$, which completes the proof.  \end{proof}

We record the following observation as a corollary.

\begin{cor} \label{cor:ext.deg_component} Let $\Lambda$ be a symmetric algebra. If there exists a nonprojective module with finite extension degree in a component $\mathcal{C}$ of $\Gamma_{\Lambda}$, then all modules in $\mathcal{C}$ have finite extension degree.
\end{cor}
\begin{proof} Over a symmetric algebra every almost split sequence has the form of the short exact sequence in Lemma \ref{lem:ed.of.middle.term}.   The lemma tells us that if a non-projective module has finite extension degree then so do all of its neighbors in $\Gamma_{\Lambda}$.  The corollary then follows from the fact that $\mathcal{C}$ is a connected graph.  
\end{proof}

In view of the above corollary, it makes sense to discuss the finiteness of the extension degree of an AR component. We write $\ed( \mathcal{C}) < \infty$ when any (and therefore every) non-projective module in the component $\mathcal{C}$ has finite extension degree. Note that there need not exist a bound on the extension degree of the modules in the component.

We now set up some notation and terminology that we use throughout this section. First, fix a $\Lambda$-module $M$ and suppose that $\mathcal{C}$, is the component of the AR quiver containing $M$.  Let $\mathcal{C}_M \subset \mathcal{C}$ be the {\it cone of $M$}, it consists of all predecessors of $M$ in $\mathcal{C}$.  For a $\Lambda$-module $X \in \mathcal{C}_M$, we denote by $d_M(X)$, the length of the shortest directed path in $\mathcal{C}$ from $X$ to $M$.  Let $\mathcal{C}^0_M = \{M\}$, and for every $n \geq 0$ set \begin{align*} \mathcal{C}^{n+1}_M = \{ X \in \mathcal{C}_M \  | \ X \ {\rm is \ an \ immediate \ predecessor \ of \ some} \ Y \in \mathcal{C}_M^n\} \end{align*}
That is, $X \in \mathcal{C}_M^d$ if and only if there exists a directed path of length $d$ from $X$ to $M$.  In particular we see that $d_M(X) = d$ is a sufficient condition for $X$ to be in $\mathcal{C}_M^d$.  If a slice of $\mathcal{C}$ does not contain a cycle, this condition is necessary. Also, note that $\mathcal{C}_M = \bigcup_{n \geq 0} \mathcal{C}_M^n$, but the union need not be disjoint.

The diagram below shows the beginning of the cone of $M$ in an arbitrary component $\mathcal{C}$.  Each $\mathcal{C}_M^n$ appears above the column of modules it contains. 
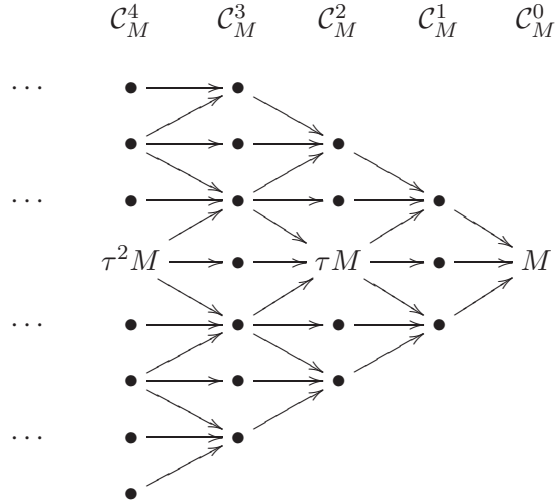
\begin{align*}\xymatrix@C=15pt@R=10pt{%
 & \mathcal{C}_M^4  &\mathcal{C}_M^3&  \mathcal{C}_M^2 & \mathcal{C}_M^1& \mathcal{C}_M^0 \\
\cdots &\bullet \ar[r]& \bullet \ar[dr]  &&   && \\ 
&   \bullet\ar[dr] \ar[ur] \ar[r] &\bullet \ar[r]& \bullet\ar[dr]  & &  & \\ 
\cdots &\bullet \ar[r]& \bullet\ar[ur]\ar[dr] \ar[r] &\bullet \ar[r]& \bullet\ar[dr]  & \\ 
& \tau^2 M \ar[ur] \ar[dr]   \ar[r] &\bullet \ar[r]& \tau M \ar[ur] \ar[dr] \ar[r] &\bullet \ar[r]& M \\ 
 \cdots &\bullet \ar[r]& \bullet\ar[ur] \ar[dr]  \ar[r] &\bullet \ar[r]& \bullet\ar[ur] \\ 
  & \bullet \ar[dr] \ar[ur]\ar[r]& \bullet \ar[r] &\bullet\ar[ur]&   &  \\
   \cdots & \bullet \ar[r] & \bullet\ar[ur] && & \\ 
  &\bullet \ar[ur]&  &&  \\%
}\end{align*}

Since the Auslander-Reiten quiver $\Gamma_{\Lambda}$ is a locally finite graph, we have that each $\mathcal{C}_M^n$ contains finitely many modules up to isomorphism.  We denote by add$\mathcal{C}_M^n$ the subset of $\Lambda$-mod consisting of all direct summands of direct sums of modules in $\mathcal{C}^n_M$.  We say that a module $X$ generates add$\mathcal{C}_M^n$ if it contains each module in $\mathcal{C}_M^n$ as a direct summand.   Then we define $\ed(\mathcal{C}_M^n) = \ed(X)$, where $X$ is any generator of add$\mathcal{C}^n_M$.  Note that this definition is independent of the generator chosen.

\begin{prop} \label{prop:ed.of.phi} Let $\Lambda$ be a symmetric algebra and assume that $M$ is a non-projective $\Lambda$-module with $\ed(M) = m$.  Then for all $d\geq 0$ we have $\ed(\mathcal{C}_M^d) = m+2d$. 
\end{prop}

\begin{proof}  We proceed by induction on $d$. There is nothing to show when $d = 0$. So, assume the claim holds when $d=k \geq 0$.

Choose a module $X=\bigoplus X_i$ that generates add$\mathcal{C}^k_M$ where each $X_i$ is indecomposable.  Denote by $E_i$ be the middle term of the almost split sequence ending in $X_i$ and set $E= \oplus E_i$.  Since a module is in $\mathcal{C}_M^{k+1}$ if and only if it is an immediate predecessor of some module in $\mathcal{C}_M^k$ and $X$ generates add$\mathcal{C}_M^k$ we see that $E$ generates add$\mathcal{C}_M^{k+1}$.  Thus, $\ed(\mathcal{C}_M^{k+1}) = \ed(E)$.  Summing the almost split sequences we obtain an exact sequence:  \begin{align*}  0 \to  \Om^2 X \to  E  \to  X \to 0 \end{align*}  The induction hypothesis gives $\ed(X) = \ed(\mathcal{C}_M^k) = m+2k$.  Finally, an application of  Lemma\autoref{lem:ed.of.middle.term} yields $\ed(E)= m+2(k+1)$ as needed. \end{proof}

In the next two results we show that a component in the AR quiver of a symmetric algebra satisfying some finiteness conditions may not contain a non-projective module with eventually vanishing self-extensions.

\begin{thm} \label{thm:tubes}
Let $\Lambda$ be a symmetric algebra and assume that $\mathcal{C}$ is a component of $\Gamma_{\Lambda}$ containing a $\tau$-periodic module. Then any module with eventually vanishing self-extensions in $\mathcal{C}$ must be projective.  
\end{thm}

\begin{proof} If $\mathcal{C}$ contains a $\tau$-periodic module, then every non-projective module in $\mathcal{C}$ is $\tau$ periodic and therefore $\Om$-periodic, as $\Lambda$ is symmetric.  Let $n$ be so that  $\Om^nX \cong X$ for all $X \in \mathcal{C}$.  Then, if $M \in \mathcal{C}$ and $\ed(M) < \infty$, for all $k \gg 0$ we have \begin{align*} 0 = \Eext^{kn}(M,M) \cong \StHom(\Om^{kn}M,M) \cong \StHom(M,M) \end{align*}  This occurs only when $M$ is projective.
\end{proof}

\begin{thm} \label{thm:bands} Let $\Lambda$ be a symmetric algebra.  Assume that the $\mathcal{C}$ is a component of $\Gamma_{\Lambda}$ that has finitely many $\tau$-orbits and that its slice does not contain an unoriented cycle. Then any module with eventually vanishing self-extensions in $\mathcal{C}$ must be projective. 
\end{thm}

\begin{proof} We may assume that $\mathcal{C}$ does not contain a $\tau$-periodic module, otherwise the previous proposition applies.  Suppose that $\mathcal{C}$ has $n$ $\tau$-orbits and $M$ is any non-projective module in $\mathcal{C}$. 

Since a slice of $\mathcal{C}$ does not contain a cycle, we have that a module $X$ is in $\mathcal{C}_M^d$ if and only if $d_M(X) = d$.  Also note that for all $d \geq 0$ there is a containment $\mathcal{C}_{\tau M}^d \subseteq \mathcal{C}_M^{d+2}$.  For small values of $d$, this containment may be strict.  Modules on a slice containing $M$ are not in the cone of $\tau M$ and these are the modules in $\mathcal{C}_M^{d+2} \backslash \mathcal{C}_{\tau M}^d$. 

Since $\mathcal{C}$ has only $n$ $\tau$-orbits, the distance from any module on a slice containing $M$ to $M$ is at most $n$.  Therefore, $\mathcal{C}_M^{n+2}$ is contained in the cone of $\tau M$ and so $\mathcal{C}_{\tau M}^n = \mathcal{C}_M^{n+2}$.

Proposition\autoref{prop:ed.of.phi} gives the equalities $\ed(\mathcal{C}_M^{n+2}) = \ed(M) + 2(n+2)$ and $\ed(\mathcal{C}_{\tau M}^n) = \ed(\tau M) + 2n$.  Since $\ed(M) = \ed(\tau M)$ and $\mathcal{C}_{\tau M}^n = \mathcal{C}_M^{n+2}$, we get $\ed(M) + 2(n+2) = \ed(M) +2n$ so that $M$ must have infinite extension degree.  \end{proof}

%-------------------------------------------stopped here 15:00 24 July 

The following observation about the cone of a module $M$ with eventually vanishing self-extensions is needed in the proof of the next proposition.  We record it separately as it provides a source of finitely generated modules over the graded artin algebra $\Ext^*(M,M)$.

\begin{lem}  \label{lem:cone}Let $\Lambda$ be a symmetric algebra.  Suppose that $M$ is a $\Lambda$-module with $\ed(M)=m$ and $Y \in \mathcal{C}_M$. Then \begin{enumerate}[(i)] %
 \item $\Eext_{\Lambda}^k(Y,M) = 0$ for $k > m$, and  
 \item $\Eext_{\Lambda}^k(M,Y) = 0$ for $k > m+2 d_M(Y)$. \end{enumerate}
 \end{lem}

\begin{proof}  We proceed by induction on $d_M(Y)$.  When this is zero the claims are clear.  So we assume that the proposition holds for all modules of distance $d \geq 0$ to $M$.  If $d_M(Y) = d+1$ then $Y$ has a neighbor, say $X$, in $\mathcal{C}_M$ with $d_M(X) = d$.  The almost split sequence ending in $X$ is of the form $ 0 \rightarrow \Om^2 X \rightarrow Y \oplus Y' \rightarrow X  \rightarrow 0$ where $Y' \in \Lmod$.  

For (i), we have $\Ext^k(X,M) = 0$ for $k>m$, by the induction hypothesis.  Part (\ref{lem:con.indices}) of Lemma \ref{lem:indices} applied to the above sequence  shows that $\Ext^k(Y,M) = 0$ for $k>m$, as claimed. 

For (ii), the induction hypothesis gives $\Ext^k(M,X) = 0$ for $k >m+ 2d$.  An application of part (\ref{lem:cov.indices}) of Lemma \ref{lem:indices} to the above sequence shows that $\Ext^k(M,Y) = 0$ for $k>m+ 2d+2$, as claimed. \end{proof}

In order to classify the extension degree of all modules in a component of stable type $\mathbb{Z}A_{\infty}$, we will need the following result. 

\begin{prop}  \label{prop:jump} Let $\Lambda$ be a symmetric artin algebra. Take a short exact sequence $0 \rightarrow \Om^2 M \rightarrow N \oplus L \rightarrow M \rightarrow 0$  in $\Lmod$ with $m = \ed(M)$ and $n= \ed(N)$. If $n<m<\infty$, then $\ed(L)=m+2$ and 
 $\Eext^{m+2}(L,L) \cong \Eext^m(M,M)$.  
\end{prop}

\begin{proof}   Proposition\autoref{lem:ed.of.middle.term} gives $\ed(L)\leq m+2$ and \begin{align*}\Eext^m(M,M) \cong \Eext^{m+2}(N\oplus L, N\oplus L). \end{align*} To complete the proof, it suffices to show that the following summands of the right hand term are all zero: $\Eext^{m+2}(N, N)$, $\Eext^{m+2}(L, N)$, and $\Eext^{m+2}(N, L)$.  

First, $\Eext^{m+2}(N,N)=0$ is immediate since $m>n=\ed(N)$. We write $N = \bigoplus N_i$ and $L = \bigoplus L_j$ where each $N_i$ and $L_j$ are indecomposable.  Note that $\ed(N_i) \leq n$ for each $i$ and we have the following subgraph of $\Gamma_{\Lambda}$ for each $i$ and $j$: \begin{align*} \xymatrix@R=2pt{N_i \ar[dr] & & \Om^{-2}N_i \\ & M \ar[dr] \ar[ur] & \\ L_j  \ar[ur] & & \Om^{-2}L_j  }\end{align*}  So we see that $L_j$ is in $\mathcal{C}_{\Om^{-2}N_i}$ and $d_{\Om^{-2}N_i} \leq 2$ for each $i$ and $j$.

Since $L_j \in \mathcal{C} _{\Om^{-2}N_i}$ we may employ Lemma\autoref{lem:cone} and then shift dimensions to get $0=\Eext^k(L_j,\Om^{-2}N_i)=\Eext^{k+2}(L_j,N_i)$ for all $k>n$, that is $\Eext^{k}(L_j,N_i)=0$ for all $k>n+2$. In particular, we get $\Eext^{m+2}(L,N)=0$.

Next, as $d_{\Om^{-2}N_i}(L_j)=2$, Lemma\autoref{lem:cone} along with a dimension shift give us $0=\Eext^k(\Om^{-2}N_i, L_j)=\Eext^{k-2}(N_i,L_j)$ for all $k>n+4$ which is to say $\Eext^k(N_i,L_j)=0$ for all $k>n+2$. Therefore, we get $\Eext^{m+2}(N,L)=0$. \end{proof}

Until this point, we have been careful to track the degrees where the vanishing of extensions begins.  This attention to detail is used repeatedly in the sequel.  In some instances, however, the following notation can simplify our arguments considerably.

\begin{defn} \label{defn:perp}  If $\Ext^i(M,N) = 0$ for all $i \gg 0$ we write $M \perp N$.  We set $^{\perp}M =\{ X \ | \ X \perp M \} $ and $M^{\perp} = \{Y \ | \ M \perp Y\}$. \end{defn}

\begin{rem} \label{rmk:perp}

(i) For each $\Lambda$-module $M$, the sets $^{\perp}M$ and $M^{\perp}$ satisfy the two out of three property on short exact sequences.  That is, if the sequence  $0 \to X \to Y \to Z \to 0$ is exact and any two modules from  $\{X,Y,Z\}$ are in $^{\perp}M$ or $M^{\perp}$, then so is the remaining module.  

(ii) If $\ed(M) < \infty$ and $N$ is any module in the same component of the Auslander-Reiten quiver as $M$, then it follows from Lemma \ref{lem:cone} that $M \perp N$ and $N \perp M$.
 \end{rem}

\begin{prop} \label{prop:perp} Let $\Lambda$ be a symmetric artin algebra. Assume that $N$ is an indecomposable $\Lambda$-module of finite extension degree.  Let $f: M \to N$ be an irreducible epimorphism (monomorphism), then  $\ker f$  (respectively, $\coker f$) has finite extension degree. \label{prop:perp} \end{prop}

\begin{proof}  Since $f$ is irreducible, the modules $M$ and $N$ are in the same component of $\Gamma_{\Lambda}$.  In particular, this gives $\ed(M) < \infty$ and $M \perp N$ and $N \perp M$.  

We prove the proposition in the case when $f$ is an epimorphism. A similar argument works when $f$ is a monomorphism.  Since $f$ is an epimorphism, we have a short exact sequence $0 \to \ker f \to M \to N \to 0$.  Since $M \perp M$ and $M \perp N$ the two out of three property gives $M \perp \ker f$.  Similarly, since $N \perp N$ and $N \perp M$, the two out of three property gives $N \perp \ker f$. One more application of the two out of three property now gives $\ker f \perp \ker f$.
\end{proof}

%%%%%%%%%%%%%%%%%%%%%%%%%%%%%%%%%%%%%%%%%%%%%%%%%%%%%%%%%%%%%%%%%%%%%%%%%%%%%%%%%%%%%%%%%%%%%%%%%%%%%%%%%%%%%%%%%%%%%%%%%%%%%%%%%%%%%%%%%%%%%%%%%%%%%%%%%%%%%%%%%%%%%%%%%%%%%%%%%%%%%%%%%%%%%%%%%%%%%%%%%%%%%%%%%%%%%%%%%%%%%%%%%

\section{Eventually $\Omega$-perfect components} \label{sec:Omega_perfect}

In order to determine the extension degree of modules in all of the possible types of AR components, we separate the components into two classes: those that are eventually $\Omega$-perfect and those that are not eventually $\Omega$-perfect. 
This section is dedicated to the study of the AR components that are eventually $\Om$-perfect. We show that any such component of finite extension degree must be stably $\mathbb{Z}A_{\infty}$. In addtion, we describe the extension degree of each module in such a component.

Recall that an irreducible map $f: M \to N$ is called {\it $\Omega$-perfect} if the induced irreducible maps $\Omega^n f: \Omega^n M \to \Omega^n N $ are either all monomorphisms, for every $n\geq0$, or are all epimorphisms, for every $n\geq0$. An indecomposable non-projective module $M$ is $\Omega$-perfect if all irreducible homomorphisms ending at $M$ and beginning in $\tau M$ are $\Omega$-perfect. We say that an AR component $\mathcal{C}$ is eventually $\Omega$-perfect if for each $M \in \mathcal{C}$ there is an $i \in \mathbb{N}$ so that $\tau^i M$ is $\Omega$-perfect. 

The concepts of an $\Omega$-perfect map and an $\Omega$-perfect module where introduced in \cite{GZ1}. The authors also showed that $\Omega$-perfect modules occur quite often. In particular, if an algebra has no periodic simple modules, then every indecomposable non-projective module is eventually $\Omega$-perfect \cite[Proposition 2.4]{GZ2}.  

Throughout the section the {\it valence} of a vertex in a stable AR component is the valence of the vertex in the graph of a slice of the stable component containing it. Equivalently, the valence denotes the number of irreducible maps ending at the vertex in the stable AR quiver.  We denote the {\it length} of the $\Lambda$-module $M$ by $\ell(M)$.

We now begin working with stable AR components that are eventually $\Om$-perfect.

\begin{prop} \label{prop:boundary.exists} Let $\Lambda$ be a symmetric algebra. Suppose that $\mathcal{C}$ is an eventually $\Omega$-perfect component in the stable AR quiver of $\Lambda$ containing a module of finite extension degree.  Let $M$ be an $\Omega$-perfect module in $\mathcal{C}$ of minimal length.  Then $\alpha(M) = 1$.\end{prop}

\begin{proof} Consider the almost split sequence ending in $M$: \begin{align*}
\xymatrix{ 0 \ar[r] & \Om^2 M \ar[rr]^{[f_1,...,f_r]^T} & & \bigoplus_1^r E_i \ar[rr]^{[g_1,...,g_r]} & & M \ar[r] & 0 } \end{align*}
where each $E_i$ is indecomposable.  Suppose for purposes of contradiction that that $r \geq 2$.

If $g_i$ is a monomorphism for some $i$, then $E_i$ is $\Omega$-perfect by \cite[Lemma 2.6]{GZ2} and $\ell(E_i) < \ell(M)$, contrary to the choice of $M$.  Thus, each $g_i$ is an epimorphism. If $r\geq 2$, then $f_i$ is also an epimorphism for each $i$, by \cite[Lemma 2.5]{GZ2}.  Since $M$ is $\Omega$-perfect, for each $n\geq1$ the induced morphisms $\Omega^{2n} f_1: \Omega^{2(n+1)}M \to \Omega^{2n}E_1$ and $\Omega^{2n} g_1: \Omega^{2n}E_1 \to \Omega^{2n} M$ are epimorphisms.  Composing these maps gives us epimorphisms $\Omega^{2n}M \to M$ for each $n \in \mathbb{N}$.  Lemma X.3.1 of \cite{ARS} then says that for each $n \in \mathbb{N}$  we have $0 \neq \StHom(\Omega^{2n}M,M) \cong \Ext^{2n}(M,M)$.  But Corollary\autoref{cor:ext.deg_component} says that $M$ must have finite extension degree. We have arrived at a contradiction. \end{proof}

Since every eventually $\Omega$-perfect component contains an $\Omega$-perfect module of minimal length, it follows that any such component of finite extension degree admits a module whose almost split sequence has an indecomposable middle term.

The next step in showing that an $\Omega$-perfect component $\mathcal{C}$ of finite extension degree is of stable type $\mathbb{Z}A_{\infty}$ is to show that $\alpha(M) \leq 2$ for all $M$ in $\mathcal{C}$.  To this end, we proceed with the following observations concerning $\Omega$-perfect modules.

\begin{lem} \label{lem:alpha<=2} Let $M \in \mathcal{C}$ be an $\Omega$-perfect module.  Let \begin{align*} \xymatrix{ 0 \ar[r] & \Om^2 M \ar[rr]^{[f_1,...,f_r]^T} & & \bigoplus_1^r E_i \ar[rr]^{[g_1,...,g_r]} & & M \ar[r] & 0 } \end{align*}
 be an almost split sequence where each $E_i$ is indecomposable.  If $f_1$ is a monomorphism, then $g_1$ is an epimorphism  and $r \leq 2$.  If $r = 2$, then $f_2$ is an epimorphism and $g_2$ is a monomorphism.
 \end{lem}
 
 \begin{proof} If $f_1$ is a monomorphism, then $g_1$ is an epimorphism by \cite[Lemma 2.5]{GZ2}, and then $r \leq 2$ by \cite[Lemma 3.4]{GZ2}.
 
 If $r =2$, then the following square is a pushout diagram: \begin{align*} \xymatrix@R=5pt{ & E_1 \ar@{->>}[dr]^{g_1}& \\ \Om^2X \ar@{^{(}->}[ur]^{f_1} \ar[dr]_{f_2} & & X \\ & E_2 \ar[ur]_{g_2} & } \end{align*}  It follows that parallel arrows are monomorphisms/epimorphisms simultaneously.  This gives that $f_2$ is an epimorphism and $g_2$ is a monomorphism. \end{proof}

\begin{cor} \label{cor:irred.mono} Suppose that $N$ is an indecomposable $\Om$-perfect module and that $f: M \to N$ is an irreducible monomorphism.  Then $M$ is $\Om$-perfect, $\alpha(M) \leq 2$ and if $\alpha(M) = 2$, the almost split sequence ending in $M$ has the form: \begin{align*} \xymatrix@R=5pt{& & \tau N \ar@{->>}[dr]^{g}& & \\ 0 \ar[r] &\tau M \ar@{^{(}->}[ur]^{\tau f} \ar@{->>}[dr]_{\alpha} & & M \ar[r]&  0 \\& & E \ar@{^{(}->}[ur]_{\beta} & & } \end{align*} \end{cor}

\begin{proof} That $M$ is $\Omega$-perfect follows from \cite[Lemma 2.6]{GZ2}. Since $N$ is $\Om$-perfect, we have $\tau f: \tau M \to \tau N$ is a monomorphism.  The remaining claims now follow immediately from the preceding lemma.\end{proof}

\begin{lem} \label{lem:alpha>=3} Suppose that $\mathcal{C}$ has finite extension degree.  Let $X \in \mathcal{C}$ be an $\Omega$-perfect module with almost split sequence: \begin{align*} \xymatrix{ 0 \ar[r] & \Om^2 X \ar[r]^{[f_i]^T} & \bigoplus_1^r E_i \ar[r]^{[g_i]} & X \ar[r] & 0 } \end{align*} where each $E_i$ is indecomposable.  If $\alpha(X) \geq 3$, then $f_i$ is an epimorphism  and $g_i$ is a monomorphism for each $i$.
\end{lem}

\begin{proof} The previous lemma gives that if $f_i$ is a monomorphism for some $i$, then $\alpha(X) \leq 2$. This shows that each $f_i$ is an epimorphism.

If $g_i$ is an epimorphism for some $i$, then since $X$ is $\Om$-perfect we obtain epimorphisms $\Om^{2n}X \to X$ for each $n \in \mathbb{N}$.  Now Lemma X.3.1 of \cite{ARS} gives $\StHom(\Om^{2n}X,X) \neq 0$, i.e. $\Eext^{2n}(X,X) \neq 0$ for all $n \in \mathbb{N}$.  This contradicts the assumption that $\mathcal{C}$ has finite extension degree.  Thus, each $g_i$ is a monomorphism, as claimed.
\end{proof}

%%%%%%%%%%%%%%%%%%%%%%%%%%%%%%%%%%%%%%

\begin{prop} \label{prop:Omega.perfect1} Let $\mathcal{C}$ be an eventually $\Om$-perfect component containing a module of finite extension degree.  Then for each non-projective module $M$ in $\mathcal{C}$ the middle term of the almost split sequence ending in $M$ has at most two non-projective indecomposable summands.\end{prop}

\begin{proof} Since $\mathcal{C}$ is eventually $\Omega$-perfect, we may apply $\tau$ repeatedly to any indecomposable module $M$ in $\mathcal{C}$ and assume that $M$ is $\Omega$-perfect and that there is no projective module in the cone $\mathcal{C}_M$.  We show that if $\alpha(M) \geq 3$ for such an $M$, then the stable part of $\mathcal{C}$ must have finitely many $\tau$-orbits and its slice is a tree. This contradicts Theorem\autoref{thm:bands}.

We now construct the slice of the stable component $\mathcal{C}$ that contains the module $M$. Begin with the almost split sequence ending in $M$: \begin{align*} \xymatrix{ 0 \ar[r] & \tau M \ar[r]^{[f_i]^T} & \bigoplus_1^r E_i \ar[r]^{[g_i]} & M \ar[r] & 0 } \end{align*} where the $E_i$ are indecomposable modules and assume that $r \geq 3$.

We know from Lemma\autoref{lem:alpha>=3} that each $g_i$ is a monomorphism. Corollary \autoref{cor:irred.mono} gives that $\alpha(E_i) \leq 2$ for each $i$. If $\alpha(E_i) =1$, then $E_i$ lies at an end of the slice.  If $\alpha(E_i) = 2$, Corollary\autoref{cor:irred.mono} implies that the almost split sequence ending in $E_i$ has the form: \begin{align} \label{eq:AlmostSplit}
\xymatrix@R=5pt{& & \tau M   \ar@{->>}[dr]^{g_i}& & \\ 0 \ar[r] &\tau E_i \ar@{^{(}->}[ur]^{\tau f_i} \ar@{->>}[dr]_{\alpha_i} & & E_i \ar[r]&  0 \\& & E_i' \ar@{^{(}->}[ur]_{\beta^1_i} & & } \end{align} where $E_i'$ is indecomposable and $\Om$-perfect. We therefore obtain exactly one new map $E_i' \longhookrightarrow E_i$ in our slice. 

If $\alpha(E_i')=1$ we have reached the end of the slice. Otherwise, we may repeat the above steps to obtain again exactly one new map $E_i'' \longhookrightarrow E_i'$ in the slice. Continuing this process yields, for each $i$, a chain of irreducible monomorphisms: \begin{align*}  \ldots \longhookrightarrow E_i'' \longhookrightarrow E_i' \longhookrightarrow E_i \longhookrightarrow M \end{align*} 
This chain of proper monomorphisms must stop eventually since the module $M$ has finite length. 

Thus, a slice containing the module $M$ consists only of modules lying along finite chains of monomorphisms that terminate in $M$. Since there are only $\alpha(M)$ such chains, the stable component $\mathcal{C}$ has only finitely many $\tau$-orbits. 

Now we show that this slice is a tree.  If the slice contained an unoriented cycle, then we could find two non-isomorphic modules $X_1$ and $X_2$ with a common immediate predecessor in the slice.  So we suppose that two modules in the slice, $X_1$ and $X_2$, do have a common immediate predecessor, $E$, in the slice.  We show that $X_1 \cong X_2$ and therefore the slice does contains no unoriented cycles.  Since every map in the slice is a monomorphism, we must have irreducible monomorphisms $E \hookrightarrow X_i$ for $i=1,2$.  Since $X_i$ is $\Omega$-perfect, we see that the irreducible maps $\Om^2 E \hookrightarrow \Omega^2 X_i$ are also monomorphisms.  However, since $E$ is in the slice, the almost split sequence ending in $E$ has the form (\ref{eq:AlmostSplit}).  Thus, there is an irreducible monomorphism from from $\Omega^2E$ to only one other module.  Therefore $X_1 \cong X_2$.   \end{proof}

%%%%%%%%%%%%%%%%%%%%%%

\begin{thm}\label{thm:OmegaPerfect} Let $\Lambda$ be a symmetric artin algebra. Let $\mathcal C$ be an eventually $\Omega$-perfect AR component of finite extension degree. Then $\mathcal C$ is of stable type $\mathbb{Z}A_{\infty}$.
\end{thm}
\begin{proof} First, Proposition \ref{prop:boundary.exists} guarantees that $\mathcal{C}$ contains a module $M$ with $\alpha(M)=1$.  Then Proposition\autoref{prop:Omega.perfect1} states that $\alpha(N) \leq 2$ for all $N \in \mathcal{C}$ and thus a slice in the stable part of $\mathcal{C}$ must have the form $A_{\infty}$ or $A_n$ for some $n$. Lastly, Theorem\autoref{thm:bands} rules out the possibility that the stable part of $\mathcal{C}$ is $\mathbb{Z}A_n$, so it must be the case that the stable part of $\mathcal{C}$ is $\mathbb{Z}A_{\infty}$.
\end{proof}

\begin{rem} K. Erdmann and O. Kerner examine the dimensions of Ext-groups of modules in quasi-serial components of the stable module category of a self-injective algebra in \cite{EK}. They give a sufficient condition (Lemma 3.6) that guarantees the eventual vanishing of self-extensions in components of stable type $\mathbb{Z}A_{\infty}$. \end{rem}

Our next result complements Theorem \ref{thm:main} by providing a description of the extension degree of all modules in components of stable type $\mathbb{Z}A_{\infty}$.  Recall that the {\it quasi-length} of a module $M$, denoted $\ql(M)$ in $\mathcal{C}$ is its distance to the boundary.  A module is {\it quasi-simple} if its quasi-length is zero.

\begin{thm}\label{thm:ZAinfty}Assume that $\mathcal{C}$ is a component of stable type $\mathbb{Z}A_{\infty}$ containing a nonprojective module of finite extension degree.  Let $M$ be a quasi-simple module in $\mathcal{C}$ and set $m=\ed(M)$. If $X\in \mathcal{C}$ and $l = \ql(X)$, then  $\ed(X) = m+ 2l$ and $\Eext^{m+2l}(X,X) \cong \Eext^m(M,M)$.\end{thm}

\begin{proof} Take a path of length $l=\ql(X)$ from $X$ to a quasi-simple module and note that all quasi-simple modules in this component are of the form $\tau^iM \cong \Om^{2i}M$ for some $i \in \mathbb{Z}$.  Since none of these quasi-simples have self-extensions, we may assume that $i=0$.  We can now say $X \in \mathcal{C}_M^l $ and so $\ed(X) \leq \ed(\mathcal{C}_M^l) = m+2l$ by Proposition\autoref{prop:ed.of.phi}.  

To finish the proof, it remains to show that $\Eext^{m+2l}(X,X) \cong \Eext^{m}(M,M)$.  We induce on $l=\ql(X)$. The case $l=0$ is clear, the case $l=1$ follows from Lemma\autoref{lem:ed.of.middle.term}.

Assume now that the result holds for all modules of quasi-length $l=k \geq 1$ and that the quasi-length of $X$ is $k+1$.  There is an almost split sequence of the form: \begin{align*} 0 \rightarrow \Om^2 Y \rightarrow Z \oplus X \rightarrow Y \rightarrow 0 \end{align*} where the quasi-lengths of $Y, \Om^2 Y$ and $Z$ are $k, k$ and $k-1$ respectively.  Since for all $i \in \mathbb{Z}$, $ \ed(\Om^{2i}X))= \ed(X)$ we may replace $X$ by $\Om^{2i}(X)$ if necessary and assume that the middle term of the AR sequence ending in $Y$ does not contain a projective summand.

The induction hypothesis says $\ed(Y) = m+2k > m+2(k-1) = \ed(Z)$ and $\Eext^{m+2k}(Y,Y) \cong \Eext^m(M,M)$.  An application of Proposition\autoref{prop:jump} now gives  $\Eext^{m+2(k+1)}(X,X) \cong \Eext^{m+2k}(Y,Y)$, which completes the proof.  \end{proof}

 \begin{rem} There do exist symmetric algebras which have non-projective modules with eventually vanishing self-extensions in $\mathbb{Z}A_{\infty}$ components. Using an argument analogous to that in Section 2 of \cite{Schulz}, one can show that the $R$-module $M$ introduced on page 1004 of \cite{LiuSchulz} has $\ed(M) = 1$.  In Example 5.7 of \cite{Liu} it is shown that $M$ is in a $\mathbb{Z}A_{\infty}$ component in the AR quiver of $R$.   According to Theorem \ref{thm:ZAinfty}, we have that for each $n \in \mathbb{N}$, $R$ admits an indecomposable module $X_n$ with $\ed(X_n) = 2n-1$.   In fact, whenever a symmetric algebra does not satisfy the Generalized Auslander-Reiten Condition it must admit indecomposable modules having arbitrarily large extension degrees as seen in Corollary \ref{cor:fails}. \end{rem}

%%%%%%%%%%%%%%%%%%%%
%%%%%%%%%%%%%%%%%%%%
%%%%%%%%%%%%%%%%%%%%

\section{Components that are not eventually $\Omega$-perfect}
\label{sec:not_Omega_perfect}

In order to complete the proof of Theorem \autoref{thm:main}, it remains to show that a component of $\Gamma_{\Lambda}$ that is not eventually $\Omega$-perfect may not contain any non-projective modules of finite extension degree.  To show this, we need the following result of O. Kerner and D. Zacharia:

\begin{lem}\label{lem:not.perfect} Let $\Lambda$ be a symmetric algebra and $\mathcal{C}$ a component of $\Gamma_{\Lambda}$. If $\mathcal{C}$ is not eventually $\Omega$-perfect, then there exists an irreducible epimorphism $f:M \to N$ with $M$ and $N$ in $\mathcal{C}$ such that $\ker f$ is an $\Omega$-periodic simple $\Lambda$-module.   \end{lem}

\begin{proof} This is shown in the proof of Proposition 2.4 of \cite{GZ2}. \end{proof}

\begin{thm} \label{thm:not.perfect} Suppose $\Lambda$ is a symmetric artin algebra and that $\mathcal{C}$ is a component in $\Gamma_{\Lambda}$.  If $\mathcal{C}$ is not eventually $\Omega$-perfect, then every non-projective module in $\mathcal{C}$ has infinite extension degree.  \end{thm}

\begin{proof} Assume that $\mathcal C$ is not eventually $\Omega$-perfect. Suppose to the contrary that $\mathcal{C}$ contains a module of finite extension degree.  Then by Corollary \ref{cor:ext.deg_component} all modules in $\mathcal{C}$ have finite extension degree.

Now Lemma \ref{lem:not.perfect} guarantees the existence of an irreducible epimorphism $f: M \to N$ between modules in $\mathcal{C}$ where $\ker f$ is an $\Omega$-periodic simple module.  Next, Proposition \ref{prop:perp} gives that $\ed(\ker f)$ is finite.  Since $\ker f$ is $\Omega$-periodic, Theorem \ref{thm:tubes} gives that $\ker f$ is projective, or equivalently, it is injective.  This contradicts the fact that $\ker f$ is the kernel of an irreducible epimorphism. \end{proof}

%%%%%%%%%%%%%%%%%%%%%%%%%%%%%%%%%%%%%%%%%%%%%%%%%%%%%
%%%%%%%%%%%%%%%%%%%%%%%%%%%%%%%%%%%%%%%%%%%%%%%%%%%%%
%%%

\section{Applications}\label{sec:applications}

In this section we prove that the Generalized Auslander-Reiten Condition and the Extension Conjecture hold for large classes of symmetric artin algebras-- those without a $\mathbb{Z}A_{\infty}$ component in the stable AR quiver and those that are of wild tilted type in the sense of \cite{EKS}. Note that in the setting of symmetric algebras, the condition (GARC) requires the modules of finite extension degree to coincide with the projective modules. Also observe that the Extension Conjecture follows if all simple modules have infinite extension degree. 

\begin{cor}  Let $\Lambda$ be a symmetric artin algebra. Then $\Lambda$ satisfies the Generalized Auslander-Reiten Condition in the following cases:
\begin{enumerate}
\item $\Lambda$ has no AR component of stable type $\mathbb{ZA}_{\infty}$.
\item  $\Lambda$ is of wild tilted type.
\item $\Lambda$ is the trivial extension of an iterated tilted algebra.
\end{enumerate}
\end{cor}

\begin{proof} (1) This is a direct consequence of Theorem \autoref{thm:main}. In this case, all nonprojective modules must have infinite extension degree.

(2) It follows from Theorem 9.6 of \cite{EKS} that all non-projective modules in AR components of stable type $\mathbb{Z}A_{\infty}$ have infinite extension degree. All other non-projective modules have infinite extension degree by Theorem \autoref{thm:main}. Thus, the modules of finite extension degree are precisely the projective modules.

(3) The wild case follows from (2).  If $\Lambda$ is not wild, then (1) applies because $\Gamma_{\Lambda}$  does not contain a component of stable type $\mathbb{Z}A_{\infty}$ by \cite{T}, \cite{TW}.
\end{proof}

We note that several classes of symmetric algebras satisfying (1) have been identified. In particular, these include all symmetric algebras of Euclidean type \cite{LS}.

As a further consequence we obtain the validity of the Extension Conjecture in the above cases as the conjecture holds for all symmetric algebras satisfying (GARC).  Theorem \autoref{thm:main} gives one additional case when the Extension Conjecture holds, even if (GARC) may fail.  Consequently we also obtain a proof of the Strong No Loop Conjecture for algebras satisfying the hypotheses of the next corollary.  Note that we do not require $\Lambda$ to be an algebra over a field as is assumed in the proof of the Strong No Loop Conjecture in \cite{ILP}. 

\begin{cor}  Let $\Lambda$ be a symmetric artin algebra.  If no simple $\Lambda$-module lies in a component of stable type $\mathbb{ZA}_{\infty}$, then $\Lambda$ satisfies the Extension Conjecture.  \end{cor}

In \cite{CH} it is shown that (GARC) is satisfied by all Noetherian rings which satisfy Auslander's Condition.  For a self-injective ring $\Lambda$, it is known that $\Lambda$ satisfies (GARC) if and only if its finitistic extension degree, $\fed(\Lambda) = \sup\{\ed(M) \ | \ed(M) < \infty\}$, is finite \cite{FED}.  For these rings, the latter condition is weaker than Auslander's Condition as it only depends on the vanishing of self-extensions.  In contrast to testing for Auslander's Condition, however, it is not known if one may compute $\fed(\Lambda)$ by considering only indecomposable $\Lambda$-modules.

 For symmetric algebras, our results give an alternative proof of the above mentioned result from \cite{FED}.  Moreover, we show that to test for (GARC) one  may consider only the indecomposable $\Lambda$-modules.

\begin{cor}\label{cor:fails}  Let $\Lambda$ be a symmetric artin algebra. $\Lambda$ satisfies the Generalized Auslander-Reiten Condition if and only if the supremum \linebreak $\sup\{ \ed(M) \ | \ M \ {\rm is \ indecomposable \ and} \  \ed(M) < \infty \}$ is finite. \end{cor}

\begin{proof} If $\Lambda$ satisfies (GARC), then the supremum is zero.  If $\Lambda$ does not satisfy (GARC), then combining Theorems \ref{thm:main} and \ref{thm:ZAinfty} one sees that there exist indecomposable $\Lambda$-modules having arbitrarily large extension degrees. \end{proof}

\end{document}